\documentclass[a4paper,12pt]{amsart}
\usepackage[english]{babel}
\usepackage{amsmath,amsthm,amssymb,amsfonts}
\usepackage{multicol}
\usepackage[titletoc]{appendix}

\usepackage[pdftex]{color}
\usepackage[bookmarks=true,hyperindex,pdftex,colorlinks, citecolor=blue,linkcolor=blue, urlcolor=blue]{hyperref}
\usepackage[dvipsnames]{xcolor}
\usepackage{mathtools}
\usepackage{graphicx}

\usepackage[normalem]{ulem}

\newtheorem{theorem}{Theorem}[section]
\newtheorem{lemma}[theorem]{Lemma}

\newtheorem{proposition}[theorem]{Proposition}
\newtheorem{corollary}[theorem]{Corollary}
\newtheorem{question}[theorem]{Question}
\theoremstyle{definition}
\newtheorem{definition}[theorem]{Definition}

\newtheorem{remark}[theorem]{Remark}
\numberwithin{equation}{section}

\newcommand{\R}{\mathbb{R}}
\newcommand{\C}{\mathbb{C}}
\newcommand{\N}{\mathbb{N}}
\newcommand{\K}{\mathbb{K}}

\newcommand{\nn}[1]{{\left\vert\kern-0.25ex\left\vert\kern-0.25ex\left\vert #1 
		\right\vert\kern-0.25ex\right\vert\kern-0.25ex\right\vert}}

\renewcommand{\geq}{\geqslant}
\renewcommand{\leq}{\leqslant}

\newcommand{\NA}{\operatorname{NA}}
\newcommand{\QNA}{\operatorname{QNA}}

\newcommand{\eps}{\varepsilon}

\begin{document}
\setcounter{tocdepth}{1}


\title{Some Remarks on the Weak Maximizing Property}

\author[Dantas]{Sheldon Dantas}
\address[Dantas]{Departament de Matemàtiques and Institut Universitari de Matemàtiques i Aplicacions de Castelló (IMAC), Universitat Jaume I, Campus del Riu Sec. s/n, 12071 Castelló, Spain. \newline
	\href{http://orcid.org/0000-0001-8117-3760}{ORCID: \texttt{0000-0001-8117-3760} } }
\email{\texttt{dantas@uji.es}}

\author[Jung]{Mingu Jung}
\address[Jung]{Department of Mathematics, POSTECH, Pohang 790-784, Republic of Korea \newline
	\href{http://orcid.org/0000-0000-0000-0000}{ORCID: \texttt{0000-0003-2240-2855} }}
\email{\texttt{jmingoo@postech.ac.kr}}

\author[Mart\'inez-Cervantes]{Gonzalo Mart\'inez-Cervantes}
\address[Mart\'inez-Cervantes]{Universidad de Murcia, Departamento de Matem\'{a}ticas, Campus de Espinardo 30100 Murcia, Spain
	\newline
	\href{http://orcid.org/0000-0002-5927-5215}{ORCID: \texttt{0000-0002-5927-5215} } }	

\email{gonzalo.martinez2@um.es}

\thanks{S. Dantas was supported by the project OPVVV CAAS CZ.02.1.01/0.0/0.0/16\_019/0000778 and by the Estonian Research Council grant PRG877. The second author was supported by NRF (NRF-2019R1A2C1003857). The third author was partially supported by \textit{Fundaci\'on S\'eneca} [20797/PI/18], \textit{Agencia Estatal de Investigaci\'on} [MTM2017-86182-P, cofunded by ERDF, EU] and by the European Social Fund (ESF) and the Youth European Initiative (YEI) under \textit{Fundaci\'on S\'eneca} [21319/PDGI/19].
	}



\begin{abstract} A pair of Banach spaces $(E, F)$ is said to have the weak maximizing property (WMP, for short) if  for every bounded linear operator $T$ from $E$ into $F$, the existence of a non-weakly null maximizing sequence for $T$ implies that $T$ attains its norm. This property was recently introduced in an article by R. Aron, D. Garc\'ia, D. Pelegrino and E. Teixeira, raising several open questions. The aim of the present paper is to contribute to the better knowledge of the WMP and its limitations. Namely, we provide sufficient conditions for a pair of Banach spaces to fail the WMP and study the behaviour of this property with respect to quotients, subspaces, and direct sums, which open the gate to present several consequences. For instance, we deal with pairs of the form $(L_p[0,1], L_q[0,1])$, proving that these pairs fail the WMP whenever $p>2$ or $q<2$.
We also show that, under certain conditions on $E$, the assumption that $(E, F)$ has the WMP for every Banach space $F$ implies that $E$ must be finite dimensional. On the other hand, we show that $(E, F)$ has the WMP for every reflexive space $E$ if and only if $F$ has the Schur property. We also give a complete characterization for the pairs $(\ell_s \oplus_p \ell_p, \ell_s \oplus_q \ell_q)$ to have the WMP by calculating the moduli of asymptotic uniform convexity of $\ell_s \oplus_p \ell_p$ and of asymptotic uniform smoothness of $\ell_s \oplus_q \ell_q$ when $1 < p \leq s \leq q < \infty$. We conclude the paper by discussing some variants of the WMP and presenting a list of open problems on the topic of the paper. 
\end{abstract}
\maketitle


\section{Introduction}

Given two Banach spaces $E$ and $F$, the pair $(E, F)$ is said to satisfy the weak maximizing property (WMP, for short) if a bounded linear operator $T$ from $E$ into $F$ is norm-attaining whenever there exists a non-weakly null maximizing sequence for $T$. 
When $E = F$, we simply say that $E$ satisfies the WMP.
By a maximizing sequence for a bounded linear operator $T$ we mean a sequence $(x_n)_{n \in \N}$ of norm one elements such that $\|T(x_n)\| \longrightarrow \|T\|$. 
This property was recently defined and studied in \cite{AGPT}, motivated by \cite[Theorem 1]{PT} which asserts that the pair $(\ell_p, \ell_q )$ has the WMP for any $1 < p < \infty$ and $1 \leq q < \infty$. This result provides a strong connection between the WMP and the theory of norm-attaining operators. Indeed, it is observed by \cite[Proposition 2.4]{AGPT} that if the pair $(E, F)$ has the WMP, then an operator $T+K$ attains its norm whenever $T\in \mathcal{L}(E,F)$ and $K\in \mathcal{K}(E,F)$ satisfies $\|T\| < \|T+K\|$, generalizing a result due to J. Kover (see \cite[Theorem 4]{K}). As a consequence, the pair $(E, F)$ has the WMP for some non-trivial space  $F$ if and only if every bounded linear functional on $E$ is norm-attaining (or, equivalently, that $E$ is reflexive by using the classical James theorem), which implies that this new property can be viewed as an extended notion of reflexivity.

Even more recent is the paper \cite{GC}, where the authors have provided a new plan of action to deal with the WMP (and its variants) by using the modulus of asymptotic uniform convexity and the modulus
of asymptotic uniform smoothness (see \cite[Theorem 3.1]{GC}). In particular, using these tools they manage to cover all the positive known examples from \cite{AGPT} and, furthermore, to present a big array of new pairs of Banach spaces satisfying the WMP. Namely, they deal with concrete Banach spaces such as James, Tsirelson and Orlicz spaces.




In this paper, we present  a natural family of pairs of Banach spaces which lack the WMP (see Theorem \ref{NA}). These, together with the behaviour of the WMP with respect to quotients and subspaces 
(see Proposition \ref{prop:stability_going_down}), provide sufficient conditions for a pair of Banach spaces to fail the WMP. Most of the results in this paper follow from a suitable combination of these two results in different settings. In particular, these results show that the WMP is a geometric condition, not invariant under isomorphisms, much  more restrictive than reflexivity (see Corollary \ref{CorIsomorphicToEll2}).
Moreover, we deal with $L_p$-spaces and give partial  answers to a question posed in \cite{AGPT} (see Theorem \ref{TheoremLpLq} and Question \ref{question:Lp}):  indeed, we prove that $(L_p[0,1], L_q[0,1])$ fails to have the WMP whenever $p > 2$ or $q<2$. We also show that, given a Banach space $E$ satisfying certain natural conditions, if $(E,F)$ has the WMP for every Banach space $F$ then $E$ is finite-dimensional (see Proposition \ref{universal1}). On the other hand, it is immediate that every operator from a reflexive Banach space into a Banach space with the Schur property is compact and, therefore, attains its norm. Thus, if $F$ has the Schur property, then $(E,F)$ has the WMP for every reflexive Banach space $E$ (see \cite[Proposition 5.2]{GC}). In Theorem \ref{universal2} it is shown that, indeed, the Schur property is characterized by this property, i.e.~that a Banach space $F$ satisfies the Schur property if and only if $(E,F)$ has the WMP for every reflexive space $E$. 
Motivated by \cite[Corollary 7]{K},  \cite[Proposition 2.4]{AGPT}, and a question posed by R. Aron in the Conference on Function Theory on Infinite Dimensional Spaces XVI held in Madrid in 2019, we show in Proposition \ref{universal3} that the inequality $\|T+K\|>\|T\|$ implies that $T+K$ attains the norm whenever $T\in \mathcal{L}(E,c_0)$ and $K\in \mathcal{K}(E,c_0)$. Since there are pairs of Banach spaces $(E,c_0)$, even with $E$ being reflexive, without the WMP, this shows that the converse of \cite[Proposition 2.4]{AGPT} does not hold. We do not know if the converse might hold for pairs $(E,F)$ with both spaces $E$ and $F$ being reflexive (see Question \ref{question:Aron}).
Next, we provide a complete characterization for the pair $(\ell_s \oplus_p \ell_p, \ell_s \oplus_q \ell_q)$ to have the WMP by calculating the moduli of asymptotic uniform convexity of $\ell_s \oplus_p \ell_p$ and asymptotic uniform smoothness of $\ell_s \oplus_q \ell_q$ when $1 < p \leq s \leq q < \infty$ (their formal definitions are given just after Theorem \ref{directsum}). Finally, we discuss some variants of the WMP and the relation among them (Proposition \ref{prop:quasiWMP}), and devote Section \ref{openquestions} to present a list of open problems on the topic.


\section{The Results} \label{section2}

Let us start this section with the basic notation and definition we will be using throughout the paper. Let $E$ and $F$ be Banach spaces over a scalar field $\mathbb{K}$, which can be the real numbers $\R$ or the complex numbers $\C$. Let $B_E$ and $S_E$ be the closed unit ball and the unit sphere of the Banach space $E$, respectively. We denote the set of all bounded linear operators from $E$ into $F$ by $\mathcal{L}(E, F)$. When $E = F$, we simply write $\mathcal{L}(E)$. $\mathcal{K}(E, F)$ stands for the set of compact operators. We say that $T \in \mathcal{L}(E, F)$ is norm-attaining or attains its norm if there exists $x_0 \in S_E$ such that $\|T(x_0)\| = \|T\| = \sup_{x \in B_E} \|T(x)\|$. We denote by $\NA(E, F)$ the set of all norm-attaining operators from $E$ into $F$. As mentioned earlier, a Banach space $E$ must be reflexive whenever a pair of Banach spaces $(E, F)$  satisfies the WMP for some non-trivial $F$ by \cite[Corollary 2.5]{AGPT}. This means that, in order to look up for new examples of pairs $(E, F)$ satisfying such a property, we always have to assume previously that $E$ is reflexive.

As a counterpart for the already mentioned paper \cite{GC}, where some conditions for a pair of Banach spaces to have the WMP are given, the next theorem provides natural pairs of Banach spaces failing such a property. By a trivial Banach space we refer to a Banach space consisting of just the vector zero.

\begin{theorem} \label{NA} Let $E, F, X, Y$ be nontrivial Banach spaces and suppose that there exists $T \not\in \NA(E, F)$. Then,  
\begin{itemize}
	\item[(a)] the pair $(E \oplus_p X, F \oplus_q Y)$ fails the WMP whenever $1 \leq q < p < \infty$;
	
	\item[(b)] the pair $(E \oplus_{\infty} X, F)$ fails the WMP.
\end{itemize}	
\end{theorem}

Before proving Theorem \ref{NA} (for its proof, see Section \ref{ProofNA}), let us point out some of its relevant consequences that we will be getting  throughout the paper. In particular, we answer \cite[Questions 3.2]{AGPT}, where the authors asked whether the WMP is preserved under isomorphisms and conjectured that the WMP for $E$ is {\it not}  equivalent to the fact that $E$ is reflexive (see Corollary \ref{CorIsomorphicToEll2}). Also, we will use Theorem \ref{NA} as a tool to deal with pairs of the form $(L_p[0,1], L_q[0,1])$ and $(\ell_s \oplus_p \ell_p, \ell_s \oplus_q \ell_q)$.

In order to provide such consequences, we will need the following elementary facts about the behavior of the WMP when it comes to closed subspaces, direct sums, and quotient spaces. 

In what follows, the symbol $\leq$ stands for closed subspaces.

\begin{proposition}\label{prop:stability_going_down}
	Let $E, F$ be Banach spaces. 
	\begin{itemize}
		\item[(a)] If $(E, F)$ has the WMP, then so does $(E, F_1)$ for every subspace $F_1 \leq F$.
		\item[(b)]  If $(E, F)$ has the WMP, then so does $(E/E_1, F)$ for every $E_1 \leq E$. In particular,  $(E_1, F)$ has the WMP whenever $E_1 \leq E$ is $1$-complemented.
		\item[(c)] If both $(E_1, F)$ and $(E_2, F)$ have the WMP, then so does $(E_1 \oplus_1 E_2, F)$.
		\item[(d)] If both $(E, F_1)$ and $(E, F_2)$ have the WMP, then so does $(E, F_1 \oplus_\infty F_2)$.
	\end{itemize} 
\end{proposition}

\begin{proof} (a) and (d) are immediate. Let us prove (b).  Let $T \in \mathcal{L}(E/E_1, F)$ be given and let $(\hat{x}_n)_{n \in \N}$ be a non-weakly null maximizing sequence in  $S_{E/E_1}$  for $T$. Consider $Q \in \mathcal{L}(E , E/E_1)$ to be the quotient map and define $\widetilde{T}:= T \circ Q \in \mathcal{L}(E, F)$. Then, $\|Q\| = 1$ and $\|\widetilde{T}\| = \|T\|$. Now, since $1 = \|\hat{x}_n\| = \inf \{ \|x\|: x \in \hat{x}_n\}$, for each $n \in \N$, we have that there exists $x_n \in \hat{x}_n$ such that $1 \leq \|x_n\| \leq 1 + 1/n$. This implies that the sequence $(y_n)_{n \in \N} := (x_n/\|x_n\|)_{n \in \N} \subset S_E$ is a non-weakly null maximizing sequence for $\widetilde{T}$ since
	\begin{equation*} 
	\| \widetilde{T}(y_n)\| = \frac{1}{\|x_n\|} \|\widetilde{T}(x_n)\| = \frac{1}{\|x_n\|} \|T(Q(x_n))\| =  \frac{1}{\|x_n\|} \|T(\hat{x}_n)\| \longrightarrow \|T\| = \|\widetilde{T}\|.
	\end{equation*} 
	Thus, $\widetilde{T}$ attains its norm at some $x \in S_E$. This implies that $\|T\| = \|\widetilde{T}\| = \| \widetilde{T}(x)\| = \|T(Q(x))\|$ and since $Q(x) \in B_{E/E_1}$, $T$ attains its norm at $Q(x)$.
	
%
	
For item (c), let us suppose that $T \in \mathcal{L} (E_1 \oplus_1 E_2, F)$ is a norm one operator for which there exists a non-weakly null maximizing sequence $(x_n,y_n)_{n \in \N} \subset S_{E_1 \oplus_1 E_2}$ but $T$ does not attain its norm. Without loss of generality, we suppose that $(x_n)_{n \in \N} \subset B_{E_1}$ is non-weakly null and $\|x_n\| \geq  \varepsilon $ for every $n \in \N$, where $\varepsilon>0$.	Notice that 
	$$ \| T(x,y)\|= \|T|_{E_1}(x)+T|_{E_2}(y)\|\leq \|T|_{E_1}\|\|x\|+\|T|_{E_2}\|\|y\|$$
	for every $(x,y)\in  E_1 \oplus_1 E_2$. Bearing in mind that, on the sphere, $\|x\|+\|y\|=1$, it follows from the previous inequality that $1=\|T\|=\|T|_{E_1}\|$ or $1=\|T\|=\|T|_{E_2}\|$. Moreover, since $\|x_n\| \geq  \varepsilon $ for every $n\in \N$, we have necessarily $1=\|T\|=\|T|_{E_1}\|$. Thus, $T|_{E_1}$ does not attain its norm. Furthermore, since the sequence $(x_n/\|x_n\|)_{n \in \N} \subset S_{E_1}$ is non-weakly null, we conclude that it cannot be a maximizing sequence for $T|_{E_1}$. Thus, there is $0\leq C<1$ such that $\|T(x_n)\|<C\|x_n\|$ for $n$ large enough.
	Since $\|T|_{E_2}\|\leq \|T\|=1$, we conclude that
\begin{eqnarray*}
\|T(x_n, y_n)\| &\leq& \|T(x_n)\|+\|Ty_n\| \\
&<& C\|x_n\|+\|y_n\| \\
&=& \|x_n\|+\|y_n\|-(1-C)\|x_n\| \\
&=& 1-(1-C)\|x_n\| \\
&\leq& 1-(1-C)\eps.
\end{eqnarray*} 
This implies that $1=\|T\| \leq 1-(1-C)\varepsilon$, a contradiction.
\end{proof}

It is worth mentioning that items (c) and (d) of Proposition \ref{prop:stability_going_down} hold also for finite $\ell_1$- and $\ell_{\infty}$-sums. Nevertheless, this is not the case for arbitrary $\ell_1$- and $\ell_{\infty}$-sums in general. Indeed, it is clear that the pair $(\R, F)$ has the WMP for every Banach space $F$ but $(\ell_1, F)$ does not. On the other hand, if the pair $(E, \ell_{\infty})$ has the WMP, then $(E,F)$ has the WMP for every separable Banach space by Proposition \ref{prop:stability_going_down}(a). We conjecture that such a Banach space $E$ must be finite-dimensional. Proposition \ref{universal1} provides some partial answers to this conjecture.

Regarding item (b) of Proposition \ref{prop:stability_going_down}, it is also worth mentioning that if $E_1$ is a closed subspace of $E$, and the pairs $(E_1, F)$ and $(E/E_1, F)$ have the WMP, then the pair $(E, F)$ does not necessarily have the WMP. This is so because the pair $(\ell_2 \oplus_p \R, \ell_2)$ fails to have it for $p > 2$ (see the proof of Corollary \ref{CorIsomorphicToEll2} below). Concerning item (a) of Proposition \ref{prop:stability_going_down}, we do not know whether the analogous case for domain spaces holds true (see Question \ref{question:subspace}).

As promised, we present the following negative results concerning \cite[Question 3.2]{AGPT} by combining Theorem \ref{NA} and Proposition \ref{prop:stability_going_down}.

\begin{corollary}  \label{CorIsomorphicToEll2} There are reflexive Banach spaces $E_1$, $E_2$, and $E_3$, all of them isomorphic to $\ell_2$, such that $(\ell_2, E_1)$, $(E_2, \ell_2)$, and $E_3$ fail the WMP.
\end{corollary}

\begin{proof} It is well-known that there exist non-norm-attaining operators in $\mathcal{L}(\ell_2)$. Thus, we can apply Theorem \ref{NA} to get that the pairs $(\ell_2, \ell_2 \oplus_q \R)$ and $(\ell_2 \oplus_{p} \R, \ell_2)$ fail the WMP whenever $2 < p \leq \infty$ and $q < 2$. The fact that $\ell_2\oplus_\infty \R$ fails the WMP now follows from Proposition \ref{prop:stability_going_down}.
\end{proof}

Let us a take a brief moment to discuss when the hypothesis of Theorem \ref{NA} is fulfilled. That is, given two Banach spaces $E$ and $F$,  when can we guarantee the existence of an operator $T \in \mathcal{L}(E, F)$ which does not attain its norm? In order to discuss the existence of such operators, recall that a Banach space $E$ is said to have the approximation property (AP, for short) if for every compact subset $K \subset E$ and every $\eps > 0$, there exists a finite-rank operator $T$ on $E$ such that $\|T(x) - x\| < \eps$ for every $x \in K$. Also, a Banach space $E$ is said to have the compact approximation property (CAP, for short) if for every compact set $K \subset E$ and every $\eps > 0$, there exists a compact operator $T$ on $E$ such that $\|T(x) - x\| < \eps$ for every $x \in K$. It is clear that the AP implies the CAP while G. Willis proved in \cite{W} that there exists a Banach space having the CAP but failing the AP.

J.R. Holub showed in \cite{H} that if $E$ and $F$ are both reflexive spaces such that $E$ or $F$ has the AP, then the condition $\mathcal{L}(E, F) = \NA(E, F)$ is equivalent to $\mathcal{L}(E, F) = \mathcal{K}(E, F)$, which in turn is equivalent to the fact that the Banach space $\mathcal{L}(E, F)$ is reflexive by using, elegantly we must say, the James Theorem. Moreover, let us point out that this equivalence is generalized by J. Mujica in \cite[Theorem 2.1]{Muj} to the case when $E$ or $F$ has the CAP instead of the AP. We will be using these facts in the next results. For further discussions on Holub and Mujica's results, we refer the reader to the recent paper \cite{DJM}, where these results are generalized and, moreover, some conditions are given so that the existence of operators that do not attain their norms is ensured.




Back to some consequences of Theorem \ref{NA}, we see that there exist operators from $L_p [0,1]$ into $L_q [0,1]$ for $1<p, q<\infty$ which do not attain their norms by combining \cite[Theorem A2]{Rosenthal-JFA1969} with \cite[Theorem 2]{H}. Thus, we can deal with pairs of the form $(L_p[0,1], L_q[0,1])$. Nevertheless, we do not know whether the pair $(L_p[0,1], L_q[0,1])$ has the WMP when $1 < p \leq 2 \leq q < \infty$ (except for $p=q=2$) (see Question \ref{question:Lp}).

\begin{theorem}
	\label{TheoremLpLq}
	If $p>2$ or $q<2$, then $(L_p [0,1],L_q [0,1])$ does not satisfy the WMP.	
\end{theorem}
\begin{proof}  
We suppose first that $p>2$. Let us notice that $\ell_2$ embeds isometrically into $L_q[0,1]$ (see, e.g., \cite[Proposition 6.4.13]{Albiac-Kalton}). Assume to the contrary that the pair $(L_p[0,1], L_q[0,1])$ satisfies the WMP. Then, by Proposition {\ref{prop:stability_going_down}}, we have that $(L_p [0,1], \ell_2)$ has the WMP. On the other hand, by using \cite[Theorem A2]{Rosenthal-JFA1969} and \cite[Theorem 2]{H}, there exists an operator from $L_p [0,1]$ to $\ell_2$ which does not attain its norm. By Theorem \ref{NA}, we have that the pair $(L_p [0,1] \oplus_p L_p [0,1], \ell_2 \oplus_2 \ell_2)$ does not have the WMP. Since $L_p [0,1] \oplus_p L_p [0,1]$ and $\ell_2 \oplus_2 \ell_2$ are isometric to $L_p [0,1]$ and $\ell_2$ respectively, we conclude that $(L_p [0,1] , \ell_2 )$ fails to have the WMP, a contradiction.

Suppose now that $q<2$. Without loss of generality, $1<p<\infty$, so $L_p[0,1]$ is reflexive. Let $Z$ be a subspace of $L_p[0,1]^*$ which is isometric to $\ell_2$ (see \cite[Theorem 6.4.13]{Albiac-Kalton}). Then, $Z^*$ is isometric to $L_p[0,1]^{**} / Z^{\perp} = L_p[0,1] / Z^{\perp} $. In particular, $L_p[0,1]$ has a quotient isometric to $\ell_2$. Suppose by contradiction that the pair $(L_p[0,1], L_q[0,1])$ satisfies the WMP. By Proposition \ref{prop:stability_going_down}.(b), the pair $(\ell_2, L_q[0,1])$ satisfies the WMP. This in turn implies that the pair $(\ell_2 \oplus_2 \R, \ell_2 \oplus_q \R)$ has the WMP, since $\ell_2 \oplus_q \R$ is a subspace of $L_q[0,1]$. This yields a contradiction due to Theorem \ref{NA}.(a).
\end{proof} 
\begin{remark}
Recall that if a measure $\mu$ is not atomic then the space $L_p(\mu)$ contains a $1$-complemented isometric copy of $L_p[0,1]$ (see, e.g., pages 210 and 211 in \cite{Rosenthal-JFA1969}).
Thus, a simple combination of Theorem \ref{TheoremLpLq} and Proposition \ref{prop:stability_going_down} yields that the pair $(L_p(\mu),L_q(\nu))$ does not satisfy the WMP whenever the measures $\mu$ and $\nu$ are not atomic and $p>2$ or $q<2$.
\end{remark}

Let us move on with more consequences of Theorem \ref{NA}. It seems to be natural to know whether there exists an infinite dimensional Banach space $E$ such that the pairs $(E, F)$ satisfy the WMP for {\it every} Banach space $F$. Analogously, we could wonder whether there exists an infinite dimensional Banach space $F$ such that the pairs $(E, F)$ satisfy the WMP for {\it every reflexive} Banach space $E$. Concerning the first question, the following result suggests that this property might imply that $E$ is finite dimensional.

\begin{proposition} \label{universal1} Let $E$ be a Banach space. Suppose that the pair $(E, F)$ satisfies the WMP for every Banach space $F$. If
\begin{itemize}
	\item[(a)] $E$ satisfies the CAP and it is isometric to $E \oplus_p E$ for some $p \in (1, \infty)$ or
	
	\item[(b)] $E^*$ has the Kadec-Klee property, i.e. if weak and norm convergent sequences coincide on $S_{E^*}$ (in particular if $E$ has Fr\'echet differentiable norm or LUR norm),
\end{itemize}
then $E$ is finite dimensional.
\end{proposition} 

\begin{proof} (a). Let us suppose that $E$ is an infinite dimensional Banach space. Then, there exists a non-compact operator on $E$ into itself; hence the result \cite[Theorem 2.1]{Muj} from J. Mujica asserts that there exists an operator $S \in \mathcal{L}(E)$ which does not attain its norm. It follows that the pair $(E \oplus_p E, E \oplus_q E)$ fails to have the WMP for every $ 1 \leq q < p$ by Theorem \ref{NA}. Since $E$ is isometric to $E \oplus_p E$, it follows that the pair $(E, E \oplus_q E)$ also fails the WMP, which is a contradiction.
	
\noindent
(b). Let $(x_n)_{n \in \N} \subset S_E$ be given. Since $E$ is reflexive, we may assume that $x_n \stackrel{w}{\longrightarrow} u$ for some $u \in B_E$. 

\vspace{0.2cm}
\noindent
{\it Claim:} We have that either $u = 0$ or $u \in S_E$.
\vspace{0.2cm}

The claim implies that $E$ must be finite dimensional, otherwise, since $E$ is reflexive, the unit sphere $S_E$ would be weakly sequentially dense in $B_E$ by Josefson-Nissenzweig theorem (see, for instance, \cite[Chapter XII, pg. 219]{D}).

Let us now prove the claim. For each $n \in \N$, let $x_n^* \in S_{E^*}$ be such that $x_n^*(x_n) = 1$.
Passing to a subsequence if necessary,  we may assume that $x_n^* \stackrel{w}{\longrightarrow} x_\infty^*$ for some $x_\infty^* \in B_{E^*}$.
 Consider the linear operator $T \in \mathcal{L}(E , \ell_{\infty})$ defined by $T(x) := (\alpha_n x_n^*(x))_{n \in \N}$ for $x \in E$, where $(\alpha_n)_{n \in \N} \subset (0, 1)$ converges to $1$. Let us notice that $\|T\| = 1$ and $(x_n)_{n \in \N}$ is a maximizing sequence for $T$. Assume that $u \not= 0$. By assumption, the pair $(E, \ell_{\infty})$ has the WMP, which implies that $T$ attains its norm. So, there exists $x_{\infty} \in S_{E}$ such that $\|T(x_{\infty})\| = 1$. This means that $\sup_{n \in \N} \alpha_n |x_n^*(x_{\infty})| = 1=|x_\infty^*(x_\infty)|$. Thus, $\|x_\infty^*\|=1$ and, by the Kadec-Klee property, $(x_n^*)_{n \in \N}$ is norm convergent to $x_\infty^*$.
 Therefore, it follows that
\begin{equation*}
|x_{\infty}^*(u)| = \lim_j |x_{\infty}^*(x_{n})|=\lim_j | x_{n}^*(x_{n})| = 1,
\end{equation*}
which implies that $\|u\| = 1$ and this proves the claim and, consequently, item (b).
\end{proof}

It is immediate from Proposition \ref{universal1}(a) that if $H$ is Hilbert space and $(H, F)$ has the WMP for every Banach space $F$, then $H$ must be finite dimensional. We do not know whether the assumption of $(E, F)$ having the WMP for every Banach space $F$ implies that $E$ is finite dimensional in general (see Question \ref{question:finitedimrange}). 

For range spaces, we have the following characterization.



\begin{theorem} \label{universal2} Let $F$ be a Banach space. Then, the following statements are equivalent. 
\begin{itemize} 	
\item[(a)] the pair $(E, F)$ satisfies the WMP for every reflexive space $E$;
\item[(b)] $F$ has the Schur property.
\end{itemize} 
\end{theorem}

\begin{proof} If $F$ has the Schur property, then every operator from a reflexive Banach space $E$ to $F$ is compact; hence the pair $(E,F)$ has the WMP for every reflexive space $E$. Conversely, assume that $F$ does not satisfy the Schur property. 
By \cite[Theorem 3.10]{DJM}, there exist a reflexive space $E$ and an operator $T: E \longrightarrow F$ which does not attain its norm. Now, Theorem \ref{NA} shows that the pair $(E \oplus_{\infty} \R, F)$ fails to have the WMP, a contradiction.
\end{proof}

It is worth mentioning that every pair of the form $(\ell_p, c_0)$ satisfies the WMP for $1<p<\infty$ (see \cite[Corollary 3.9]{GC}). However, as the space $c_0$ does not have the Schur property, it follows from Theorem \ref{universal2} that there exists an infinite dimensional reflexive space $E$ such that $(E, c_0)$ fails to have the WMP. Moreover, it is known by \cite[Lemma 2.2]{MMP} (see also \cite[Proposition 4.3]{CCJM}) that if $E$ is an infinite dimensional space, then there exists an operator $T$ from $E$ into $c_0$ which does not attain its (quasi-) norm (see the comments preceding Definition \ref{DefinitionQUASIWMP} for the definition of quasi norm-attaining operators). 
On the other hand, we shall see that for any reflexive space $E$, if $T \in \mathcal{L}(E,c_0)$ and $K \in \mathcal{K} (E,c_0)$ satisfies that $\|T \| < \|T+K\|$, then $T + K$ attains its norm (see Question \ref{question:Aron}). 



\begin{proposition} \label{universal3}
Let $E$ be an infinite dimensional reflexive space. If $T \in \mathcal{L}(E,c_0)$ and $K \in \mathcal{K} (E,c_0)$ satisfy that $\|T \| < \|T+K\|$, then $T + K$ attains its norm. However, there exists an isomorphic copy $E'$ of $E$ such that the pair $(E', c_0)$ fails to have the WMP.
\end{proposition}

\begin{proof}
Notice that if $T \in \mathcal{L} (E, c_0)$, then $T (x) = ( (T^* (e_n^*) )(x) )_{n \in \N}$ for every $x \in E$, where $(e_n^*)_{n \in \N} \subset \ell_1$ is the unit vector basis and $(T^*(e_n^*))_{n \in \N}$ is a weakly-null sequence. Moreover, as $K \in \mathcal{L}(E,c_0)$ is compact (hence, so is its adjoint $K^*$), we have that $(K^* (e_n^*))_{n \in \N}$ converges to $0$ in norm. Suppose that $\|T \| < \| T + K \|$. From the facts that $\|T + K \| = \sup_{n \in \N} \| (T^* + K^*)(e_n^*) \|$ and that $(K^* (e_n^*))_{n \in \N}$ converges in norm to zero it follows the existence of $n_0 \in \N$ such that $\|T + K \| = \| (T^* + K^*)(e_{n_0}^*) \|$. By the reflexivity of $E$, there is $x_0 \in S_E$ such that $\| (T^* + K^*)(e_{n_0}^*) \| = | ((T^* +K^* ) (e_{n_0}^*) )(x_0) |$, which implies that $T+ K$ attains its norm at $x_0$. For the second part, as $\mathbb{K}$ is 1-complemented in $E$, we have that $E$ is isomorphic to $E' := E_1 \oplus_{\infty} \mathbb{K}$ for some infinite dimensional Banach space $E_1$. Let $T: E_1 \longrightarrow c_0$ be a norm-one operator which does not attain its norm (see \cite[Lemma 2.2]{MMP}). By Theorem \ref{NA}, we get that the pair $(E', F)$ fails to have the WMP.
\end{proof}

From Proposition \ref{universal3}, we have the following consequence.

\begin{corollary}
	Let $E$ be a Banach space. If $(E', F)$ has the WMP for every isomorphic copy $E'$ of $E$ and for every Banach space $F$, then $E$ must be finite dimensional. 
\end{corollary}

Motivated from the fact \cite[Theorem 1]{PT} that the pair $(\ell_p, \ell_q)$ has the WMP for every $1 < p, q < \infty$, we wonder whether the pairs of the form $(\ell_s \oplus_p \ell_p, \ell_s \oplus_q \ell_q)$ have the WMP or not for $1 < p, q, s < \infty$. As a matter of fact, as a direct consequence of Theorem \ref{NA}, we have that the pair $(\ell_s \oplus_p \ell_p, \ell_s \oplus_q \ell_q)$ fails the WMP whenever $p > q$ and $1 < s < \infty$. Moreover, the pair $(\ell_s \oplus_p \ell_p, \ell_s \oplus_q \ell_q)$ does not have the WMP whenever $s < p \leq q$. Indeed, if it has the property, then so does $(\ell_s \oplus_p \ell_p, \ell_s)$. This implies that $(\ell_s \oplus_p \R, \ell_s \oplus_s \R)$ has the WMP, which contradicts Theorem \ref{NA}. Similarly, one can show that if $1 < p \leq q < s$, then the pair $(\ell_s \oplus_p \ell_p, \ell_s \oplus_q \ell_q)$ does not have the WMP. Consequently, we have the following result. 

\begin{corollary} \label{directsumnegative} 
The pair $(\ell_s \oplus_p \ell_p, \ell_s \oplus_q \ell_q)$ fails the WMP whenever 
\begin{enumerate}
\item $1<q < p<\infty$ and $1 < s < \infty$;
\item $1 < s < p \leq q < \infty$;
\item $1 < p \leq q < s < \infty$. 
\end{enumerate} 
\end{corollary} 

Let us notice that the case $1 < p \leq s \leq q < \infty$ is not covered by Corollary \ref{directsumnegative}. Indeed, we shall see that the pair $(\ell_s \oplus_p \ell_p, \ell_s \oplus_q \ell_q)$ does have the WMP for $1 < p \leq s \leq q < \infty$ (considering those spaces as real Banach spaces) and we can provide a complete characterization for the pairs $(\ell_s \oplus_p \ell_p, \ell_s \oplus_q \ell_q)$ by calculating the moduli of asymptotic uniform convexity and smoothness when $1 < p \leq s \leq q < \infty$. We have the following theorem.

\begin{theorem} \label{directsum} The pair of real Banach spaces $(\ell_s \oplus_p \ell_p, \ell_s \oplus_q \ell_q)$ satisfies the WMP if and only if $1 < p \leq s \leq q < \infty$. 
\end{theorem}

In order to prove Theorem \ref{directsum}, let us present the necessary definitions we need for such. For a {\it real} Banach space $E$, let us consider the following moduli. Let $t > 0$ and $x \in S_E$ be fixed. 
\begin{equation*}
\overline{\delta}_E (t, x) := \sup_{\dim(E/F) < \infty} \inf_{y \in S_F} \|x + ty\| - 1 \ \ \ \mbox{and} \ \ \ \overline{\rho}_E (t, x):= \inf_{\dim(E/F) < \infty} \sup_{y \in S_F} \|x + ty\| - 1,
\end{equation*} 
where $F$ is a finite codimensional subspace of $E$. Then the modulus of asymptotic uniform convexity of $E$ is given by
\begin{equation*}
\overline{\delta}_E(t) = \inf_{x \in S_E} \overline{\delta}_E (t, x)
\end{equation*} 
and the modulus of asymptotic uniform smoothness of $E$ is given by
\begin{equation*}
\overline{\rho}_E(t) = \sup_{x \in S_E} \overline{\rho}_E(t, x).
\end{equation*}
We refer the reader to \cite{JLPS, M} for a background of the moduli of asymptotic uniform convexity and smoothness. The following result seems to be new in the literature.


\begin{theorem} \label{ellp-sum-modulus} Let $1 < p \leq s \leq q < \infty$. Then, 
\begin{equation*} \overline{\delta}_{\ell_s \oplus_p \ell_p}(t) = \overline{\rho}_{\ell_s \oplus_q \ell_q}(t) = (1 + t^s)^{\frac{1}{s}} - 1
\end{equation*} 
for every $t > 0$. In particular, $\overline{\rho}_{\ell_s \oplus_q \ell_q}(t) > t - 1$ for every $t > 1$.
\end{theorem}

Once this result is established, Theorem \ref{directsum} is proved automatically as a consequence of \cite[Theorem 3.1]{GC}. In order to prove Theorem \ref{ellp-sum-modulus}, we need the following two straightforward lemmas.

\begin{lemma} \label{ellp-sum-lemma1} Let $r > 1$, $s > 0$, $A >0$ and $B>0$  be given. Let $\eps > 0$. Consider the continuous function $f: [0,1] \longrightarrow \R$ defined by
	\begin{equation*}
	f(t) = [(A + \eps^s t)^{\frac{r}{s}} + \eps^r(1 - t^{\frac{r}{s}}) + B]^{\frac{1}{r}}.
	\end{equation*}
Therefore,
\begin{itemize}
	\item[(i)] if $1 < r < s$, then $f$ attains its minimum at $t = 1$.
	\item[(ii)] if $r > s$, then $f$ attains its maximum at $t = 1$.
	\item[(iii)] if $r = s$, then $f \equiv (A + B + \eps^s)^{\frac{1}{s}}$.
\end{itemize}
\end{lemma}

\begin{proof} For every $t \in [0,1]$, let us set 
	\begin{equation*}
	h(t) :=   \eps^s (A + \eps^st)^{\frac{r}{s} - 1}  - \eps^r  t^{\frac{r}{s} - 1}  .
	\end{equation*}
	Then, 
\begin{equation*}
f'(t) = h(t) \cdot \frac{1}{s} \left[ (A + \eps^s t)^{\frac{r}{s}} + \eps^{r} (1-t^{\frac{r}{s}}) + B \right]^{\frac{1}{r}-1}.
\end{equation*}
Notice that 
\begin{equation*}
\eps^{r - s} \cdot t^{\frac{r}{s} - 1} = \eps^{r-s} \cdot t^{\frac{r - s}{s}} = (\eps^s \cdot t)^{\frac{r-s}{s}};
\end{equation*}
hence 
\begin{equation*}
h(t) = \eps^s \left[ (A + \eps^s t)^{\frac{r-s}{s}} - (\eps^s t)^{\frac{r-s}{s}} \right].
\end{equation*}
From this, we can conclude items (i) and (ii). Item (iii) is immediate.
\end{proof}

Similar calculations yield the following result. 

\begin{lemma} \label{ellp-sum-lemma2} Let $r > 1$, $s > 0$, and $\eps > 0$. Consider the continuous function $g: [0,1] \longrightarrow \R$ defined by
	\begin{equation*}
	g(t) := [ (t + \eps^s)^{\frac{r}{s}} + (1 - t^{\frac{r}{s}})]^{\frac{1}{r}}.
	\end{equation*}
Therefore,
\begin{itemize}
\item[(i)] if $1 < r < s$, then $g$ attains its minimum at $t=1$.
\item[(ii)] if $r > s$, then $g$ attains its maximum at $t = 1$.
\item[(iii)] if $r = s$, then $g \equiv (1 + \eps^s)^\frac{1}{s}$. 
\end{itemize}
\end{lemma}

Now we are ready to prove Theorem \ref{ellp-sum-modulus}.

\begin{proof}[Proof of Theorem \ref{ellp-sum-modulus}] Let $1 < p \leq s \leq q < \infty$. Let us show that $\overline{\delta}_{\ell_s \oplus_p \ell_p}(t) \geq (1 + t^s)^{\frac{1}{s}} - 1$. Let $t > 0$ be given. Fix $n \in \N$ and consider the element
\begin{equation*}
x_0 := ((\alpha_1, \ldots, \alpha_n, 0, 0, \ldots), (\beta_1, \ldots, \beta_n, 0, 0, \ldots)) \in S_{\ell_s \oplus_p \ell_p}. 
\end{equation*}
This implies that
\begin{equation*}
1 = \|x_0\|_{\ell_s \oplus_p \ell_p} = (|\alpha_1|^s + \ldots + |\alpha_n|^s)^{\frac{p}{s}} + |\beta_1|^p + \ldots + |\beta_n|^p.
\end{equation*}
Set $A := |\alpha_1|^s + \ldots + |\alpha_n|^s$ and $B:= |\beta_1|^p + \ldots + |\beta_n|^p$. Then, we have that $A^{\frac{p}{s}} + B = 1$, i.e., $B = 1 - A^{\frac{p}{s}}$. Consider the following finite codimensional subspace of $\ell_s \oplus_p \ell_p$ given by
\begin{equation*}
Y_0 := \left\{ z \in \ell_s \oplus_p \ell_p : z = ((\underbrace{0, \ldots, 0}_{n\text{-times}}, a_1, a_2, \ldots), (\underbrace{0, \ldots, 0}_{n\text{-times}}, b_1, b_2,\ldots)) \text{ for } a_i, b_i \in \mathbb{R}, i \in \N \right\}. 
\end{equation*} 
Then, $Y_0$ is a subspace of codimension $2n$. Let 
\begin{equation*}
z = ((0, \ldots, 0, a_1, a_2, \ldots), (0, \ldots, 0, b_1, b_2, \ldots)) \in S_{Y_0}
\end{equation*}
be given. Now notice that
\begin{equation*}
1 = \|(0, \ldots, 0, a_1, a_2, \ldots)\|_s^p + \|(0, \ldots, 0, b_1, b_2, \ldots)\|_p^p = \left( \sum_{j=1}^{\infty} |a_j|^s \right)^{\frac{p}{s}} + \sum_{j=1}^{\infty} |b_j|^p 
\end{equation*}
and that
\begin{eqnarray*}
\|x_0 + t z\|_{\ell_s \oplus_p \ell_p} &=& \left( \| (\alpha_1, \ldots \alpha_n, t a_1, t a_2, \ldots)\|_s^p + \|(\beta_1, \ldots, \beta_n, t b_1, t b_2, \ldots)\|_p^p \right)^{\frac{1}{p}} \\
&=& \left[ \left( \sum_{i=1}^n |\alpha|_i^s + t^s \sum_{j=1}^{\infty} |a_j|^s \right)^{\frac{p}{s}} + \left( \sum_{i=1}^n |\beta_i|^p + t^p \sum_{j=1}^{\infty} |b_j|^p \right) \right]^{\frac{1}{p}} \\
&=& \left[ \left( A + t^s \sum_{j=1}^{\infty} |a_j|^s \right)^{\frac{p}{s}} +  B + t^p \left(1 - \left(\sum_{j=1}^{\infty} |a_j|^s\right)^{\frac{p}{s}} \right)  \right]^{\frac{1}{p}}	
\end{eqnarray*} 
Now, by Lemma \ref{ellp-sum-lemma1}, we have that
\begin{eqnarray*}
\inf_{z \in S_{Y_0}} \|x_0 + t z\|_{\ell_s \oplus_p \ell_p} &=& \inf_{0 \leq \xi \leq 1} \left[ (A + t^s \xi)^{\frac{p}{s}} + B + t^p(1 - \xi^{\frac{p}{s}}) \right]^{\frac{1}{p}} \\
&=& [(A + t^s)^{\frac{p}{s}} + B]^{\frac{1}{p}} \\
&=& [(A + t^s)^{\frac{p}{s}} + (1 - A^{\frac{p}{s}})]^{\frac{1}{p}}
\end{eqnarray*}
and applying Lemma \ref{ellp-sum-lemma2}, we get that
	\begin{equation*}
	\inf_{z \in S_{Y_0}} \|x_0 + tz\|_{\ell_s \oplus_p \ell_p} \geq (1 + t^s)^{\frac{1}{s}}
	\end{equation*}
Therefore, for every $x \in S_{\ell_s \oplus_p \ell_p}$ with finite support, we have that
\begin{equation*}
\overline{\delta}_{\ell_s \oplus_p \ell_p}(t, x) = \sup_{\dim(X / Y) < \infty} \inf_{z \in S_Y} \|x + tz\|_{\ell_s \oplus_p \ell_p} - 1 \geq (1 + t^s)^{\frac{1}{s}} - 1.
\end{equation*}
Since $\overline{\delta}_{\ell_s \oplus_p \ell_p}(t, x) = \lim_{n \rightarrow \infty} \overline{\delta}_{\ell_s \oplus_p \ell_p} (t, P_n(x))$, where $P_n$ is the projection onto the span of the first $n$ coordinates (see \cite[Theorem 3.1, pg. 112]{M}), we have that $\overline{\delta}_{\ell_s \oplus_p \ell_p}(t, x) \geq (1 + t^s)^{\frac{1}{s}} - 1$ for every $x \in S_{\ell_s \oplus_p \ell_p}$. Therefore,
\begin{equation*}
\overline{\delta}_{\ell_s \oplus_p \ell_p}(t) \geq (1 + t^s)^{\frac{1}{s}} - 1.
\end{equation*} 
On the other hand, {\it mutatis mutandis}, we can prove that $\overline{\rho}_{\ell_s \oplus_q \ell_q}(t) \leq (1 + t^s)^{\frac{1}{s}} - 1$. Finally, since $\ell_s \oplus_p \ell_p$ and $\ell_s \oplus_q \ell_q$ contain an isometric copy of $\ell_s$, we have that $\overline{\delta}_{\ell_s \oplus_p \ell_p} (t) \leq \overline{\delta}_{\ell_s}(t) = (1 - t^s)^{\frac{1}{s}} - 1$ and $\overline{\rho}_{\ell_s \oplus_q \ell_q}(t) \geq \overline{\rho}_{\ell_s}(t) = (1 + t^s)^{\frac{1}{s}} - 1$ for every $t > 0$. These give us the desired equalities.
\end{proof}





\bigskip

Let us finish this section by discussing some variants of the WMP. In \cite{GC}, the authors considered the bidual-WMP, that is, a pair $(E, F)$ of Banach spaces has the bidual-WMP if for every operator $T \in \mathcal{L}(E, F)$, the existence of a non-weakly null maximizing sequence for $T$ implies that $T^{**} \in \NA(E^{**}, F^{**})$ (see \cite[Definition 5.13]{GC}). This seems to be natural to be considered due to Lindenstrauss result \cite[Theorem 1]{L}, where he proved that the set of all operators in which their second adjoints are norm-attaining is dense in $\mathcal{L}(E, F)$. 

We introduce a property which is formally in the middle between WMP and bidual-WMP. Recall that an operator $T \in \mathcal{L}(E, F)$ quasi attains its norm, recently defined in \cite{CCJM}, if there exists a sequence $(x_n)_{n \in \N}$ in $S_E$ such that $T(x_n) \longrightarrow u$ in norm for some $u \in F$ with $\|u\| = \|T\|$. We denote by $\QNA(E, F)$ the set of all quasi norm-attaining operators from $E$ into $F$. In general, we have that $\NA(E, F) \subset \QNA(E, F) \subset \mathcal{L}(E, F)$. Furthermore, we have that $\mathcal{K}(E, F) \subset \QNA(E, F)$.

\begin{definition}
\label{DefinitionQUASIWMP} We say that the pair $(E, F)$ has the quasi weak maximizing property (quasi-WMP, for short) if for every $T \in \mathcal{L}(E, F)$, the existence of a non-weakly null maximizing sequence for $T$ implies that $T \in \QNA(E, F)$.
\end{definition}

The following result shows some relation between the WMP, the bidual-WMP and the quasi-WMP.

\begin{proposition}\label{prop:quasiWMP} Let $E$ and $F$ be Banach spaces. 
	
\begin{itemize}
	\item[(a)] If the pair $(E, F)$ has the quasi-WMP, then it also has the bidual-WMP.
	\item[(b)] If $F$ is reflexive, then $(E, F)$ has the quasi-WMP if and only if it has the bidual-WMP.
\end{itemize}
Moreover, if $E$ is reflexive, then $(E, F)$ has the
\begin{itemize}
	\item[(c)] WMP if and only if it has the quasi-WMP if and only if it has the bidual-WMP.
\end{itemize}
\end{proposition}

\begin{proof}
Notice that if $T \in \QNA(E, F)$, then $T^* \in \NA(F^*, E^*)$ (see \cite[Proposition 3.3]{CCJM}) and then $T^{**} \in \NA(E^{**}, F^{**})$. This proves (a). On the other hand, if $T \in \mathcal{L}(E,F)$ is weakly compact and $T^{**}$ attains its norm, then $T \in \QNA(E, F)$ (see \cite[Proposition 5.1]{CCJM}); hence (b) is proved. Next, if $E$ is reflexive, then $\NA(E, F) = \QNA(E, F)$ (see \cite[Proposition 4.1]{CCJM}) and if $T$ attains its norm, then so does $T^{**}$. 
\end{proof}


It is natural to ask whether the bidual-WMP or the quasi-WMP of a pair $(E,F)$ implies the reflexivity of $E$. In general, if $\mathcal{L}(E, F) = \mathcal{K}(E, F)$ then $(E,F)$ has the quasi-WMP (and therefore the bidual-WMP), and this is the case for some pairs with $E$ being nonreflexive, such as $(E,F)$ with $E$ nonreflexive and $F$ finite dimensional or $(C(K), \ell_p)$ whenever $1 \leq p < 2$ (see, for example, \cite{Albiac-Kalton} or \cite{Rosenthal-JFA1969}).

%
%


%

\section{Proof of Theorem \ref{NA}} \label{ProofNA}

We dedicate this section for the proof of Theorem \ref{NA}. To do so, we need first the following straightforward lemma.

\begin{lemma} \label{max} Let $1 \leq q < p < \infty$. The function $f: [0, 1] \longrightarrow \R$ defined by
	\begin{equation*}
	f(a) := \sqrt[q]{(1 - a^p)^{\frac{q}{p}} + a^q} \ \ \ (a \in [0,1])
	\end{equation*}
	attains its absolute maximum at some $m \in (0, 1)$ with $f(m) > 1$.
\end{lemma}

\begin{proof} Let us notice that $f$ is continuous on the compact interval $[0, 1]$ and $f(0) = f(1) = 1$. Moreover, since $p > q$, we have that $1/q - 1/p > 0$ and then
	\begin{equation*}
	f \left( \frac{1}{2^{\frac{1}{p}}} \right) = \sqrt[q]{ \left(1 - \frac{1}{2} \right)^{\frac{q}{p}} + \left(\frac{1}{2^{\frac{1}{p}}} \right)^q} = \sqrt[q]{ 2 \left(\frac{1}{2}\right)^{\frac{q}{p}}} = 2^{\frac{1}{q} - \frac{1}{p}} > 1.
	\end{equation*}
	Therefore, the absolute maximum is attained at some $m \in (0, 1)$ and $f(m) > 1$.	
\end{proof}

Now, we are ready to prove Theorem \ref{NA}.

\begin{proof}[Proof of Theorem \ref{NA}] It is enough to prove that the pair $(E \oplus_p \K, F \oplus_q \K)$ does not satisfy the WMP by using Proposition \ref{prop:stability_going_down}. For simplicity, we assume that $\K = \R$ (the case when $\K = \C$ is proved similarly). 
Let us first assume that $1 \leq q < p < \infty$ and we prove the case $p = \infty$ and $1 \leq q < \infty$ afterwards. 
	
	Let $1 \leq q < p < \infty$. Suppose that there is $T \not\in \NA(E, F)$ with $\|T\| = 1$. Let us define the bounded linear operator $S: E \oplus_p \R \longrightarrow F \oplus_q \R$ given by
	\begin{equation*}
	S(x, a) := (T(x), a) \ \ \ ((x, a) \in E \oplus_p \R).
	\end{equation*}
	Consider the function $f$ defined in Lemma \ref{max} and $m \in (0, 1)$ where $f$ attains its absolute maximum with $f(m) > 1$. We finish the first part of this proof once we prove the following three claims.

\vspace{0.2cm}		
	\noindent
	{\bf Claim 1:} $\|S\| \leq f(m)$.
\vspace{0.2cm}		
	
	Indeed, since $\|T\| = 1$, we have that
	\begin{eqnarray*}
		\|S\| = \sup \{ \|S(x, a)\|_q: \|(x, a)\|_p = 1 \} 
		&=& \sup \{ \|(T(x), a)\|_q: \|(x, a)\|_p = 1 \} \\
		&=& \sup \{ \sqrt[q]{\|T(x)\|^q + |a|^q}: \|x\|^p + |a|^p = 1 \}. \\
		&\leq& \sup \{ \sqrt[q]{\|T\|^q\|x\|^q + |a|^q} : \|x\|^p = 1 - |a|^p \} \\
		&=& \sup \{ \sqrt[q]{(1 - |a|^p)^{\frac{q}{p}} + |a|^q}: 0 \leq |a| \leq 1 \} \\
		&=& f(m). 
	\end{eqnarray*}

\vspace{0.2cm}		
	\noindent
	{\bf Claim 2:} For every maximizing sequence $(x_n)_{n \in \N}$ for $T$, we have that $((m x_n, \sqrt[p]{1-m^p}))_{n \in \N}$ is a non-weakly null maximizing sequence for $S$.
\vspace{0.2cm}

	Indeed, let us first notice that the sequence $((m x_n, \sqrt[p]{1-m^p}))_{n \in \N}$ is non-weakly null since $\sqrt[p]{1 - m^p} \not= 0$. Let us notice also that $(m x_n, \sqrt[p]{1-m^p}) \in S_{E \oplus_p \R}$ for every $n \in \N$ since
	\begin{equation*}
	\|(m x_n, \sqrt[p]{1 - m^p})\|_p^p = \|m x_n\|^p + 1 - m^p = m^p + 1 - m^p = 1.
	\end{equation*}
	Finally, since
	\begin{equation*}
	\|S(mx_n, \sqrt[p]{1 - m^p})\|_q = \|(T(mx_n), \sqrt[p]{1 - m^p})\|_q = \sqrt[q]{m^q \|T(x_n)\|^q + (1 - m^p)^{\frac{q}{p}}} 
	\end{equation*}
	and $\|T(x_n)\| \longrightarrow \|T\| = 1$ as $n \rightarrow \infty$, we have that
	\begin{equation*}
	\|S(mx_n, \sqrt[p]{1 - m^p})\|_q \stackrel{n \rightarrow \infty}{\longrightarrow} \sqrt[q]{m^q \|T\|^q + (1 - m^p)^{\frac{q}{p}}} = \sqrt[q]{m^q + (1 - m^p)^{\frac{q}{p}}} = f(m).
	\end{equation*}
	This shows that $\|S\| \geq f(m)$ and by Claim 1, $\|S\| = f(m)$. Moreover, $((m x_n, \sqrt[p]{1-m^p}))_{n \in \N}$ is a maximizing sequence for $S$, which is non-weakly null.

\vspace{0.2cm}	
	
	\noindent
	{\bf Claim 3:} $S$ does not attain the norm. 
\vspace{0.2cm}		
	
	Let us notice first that any norm-one element in $E \oplus_p \R$ can be written as $(\lambda x, \pm \sqrt[p]{1 - \lambda^p})$ with $x \in S_E$ and $0 \leq \lambda \leq 1$. By contradiction, suppose that $S$ is norm-attaining. Thus, there exist $x \in S_E$ and $0 \leq \lambda \leq 1$ such that $\|S(\lambda x, \pm \sqrt[p]{1 - \lambda^p})\|_q = \|S\| = f(m)$ (notice that, in this case, we must have necessarily $\lambda \not= 0$ since $\|S\| = f(m) > 1$). So, we have
	\begin{eqnarray*} 
		\|S(\lambda x, \pm \sqrt[p]{1 - \lambda^p})\|_q &=& \|(T(\lambda x), \pm \sqrt[p]{1 - \lambda^p})\|_q \\
		&=& \sqrt[q]{\lambda^q \|T(x)\|^q +(1 - \lambda^p)^{\frac{q}{p}}} \\
		&=& f(m).
	\end{eqnarray*} 
	Nevertheless, since $\|T(x)\| < \|T\| = 1$, we have that
	\begin{equation*}
	\sqrt[q]{ \lambda^q \|T(x)\|^q + (1 - \lambda^p)^{\frac{q}{p}}} < \sqrt[q]{\lambda^q + (1 - \lambda^p)^{\frac{q}{p}}} \leq f(m),
	\end{equation*}
	which is a contradiction.
	
\vspace{0.1cm}
	
	Now, let us consider the case $p = \infty$. Again, by Proposition \ref{prop:stability_going_down}, it is enough to prove that the pair $(E \oplus_{\infty} \R, F)$ fails the WMP. By assumption, we can pick an operator $T \in \mathcal{L}(E, F)$ which does not attain its norm. Define $\widetilde{T}: E \oplus_{\infty} \R \longrightarrow F$ by $\widetilde {T} (x, a) := T(x)$ for every $(x, a) \in E \oplus_{\infty} \R$. Then, $\|\widetilde{T}\| = \|T\|$ and $\widetilde{T}$ does not attain its norm. Nevertheless, if $(x_n)_{n \in \N} \subset S_E$ is a maximizing sequence for $T$, then $(x_n, 1)_{n \in \N} \subset S_{E \oplus_{\infty} \R}$ is a maximizing sequence for $\widetilde{T}$ which is non-weakly null. So, the pair $(E \oplus_{\infty} \R, F)$ fails the WMP.
\end{proof}

\section{Open Questions} \label{openquestions}

In this last section, we present and discuss shortly a list of problems on the weak maximizing property.

In Proposition \ref{prop:stability_going_down}, we prove that if the pair of Banach spaces $(E, F)$ has the WMP and $F_1$ is any subspace of $F$, then so does $(E, F_1)$. It seems that every known pair $(E,F)$ with the WMP satisfies that $(E_1,F)$ also has the WMP for every subspace $E_1 \leq E$. We wonder whether this is always the case.

\begin{question} \label{question:subspace} Let $E,~F$ be Banach spaces and $E_1$ be a closed subspace of $E$. If $(E,F)$ has the WMP, does it also $(E_1,F)$? 
\end{question}

It follows from Theorem \ref{TheoremLpLq} that if there is a pair of the form $(L_p[0,1],L_q[0,1])$ with the WMP (different from the pair $(L_2[0,1],L_2[0,1])$), then $p$ and $q$ must satisfy the inequalities $1<p\leq 2 \leq q < \infty$. As a matter of fact, in \cite{GC}, the authors prove that there exist renormings $|\cdot|_p$ and $|\cdot|_q$ on $L_p[0,1]$ and $L_q[0,1]$, respectively, such that the pair $((L_p[0,1], |\cdot|_p), (L_q[0,1], |\cdot|_q))$ satisfies the WMP for $1 < p \leq 2 \leq q < \infty$. Nevertheless, we do not know what happens for the standard $p$- and $q$-norms in this case. We highlight this question.

\begin{question} \label{question:Lp} Suppose that $1 < p \leq 2 \leq q < \infty$. Does the pair of Banach spaces $(L_p[0,1], \|\cdot\|_p, L_q[0,1], \|\cdot\|_q)$ have the WMP?
\end{question}


\bigskip

In Proposition \ref{universal1}, we prove that, in several situations, the assumption that the pair $(E, F)$ has the WMP for every Banach space $F$ implies that $E$ must be finite dimensional. We wonder if this is always the case.

\begin{question}  \label{question:finitedimdomain}  \label{question:finitedimrange} Let $E$ be a reflexive Banach space. If $(E, F)$ has the WMP for every Banach space $F$, then is $E$ finite dimensional?
\end{question}

%
%



In \cite[Proposition 2.4]{AGPT}, it was shown that if a pair $(E, F)$ of Banach spaces has the WMP, $K \in \mathcal{K}(E, F)$, $T \in \mathcal{L}(E, F)$, and $\|T\| < \|T + K\|$, then $T + K$ is norm-attaining. Moreover, we observed in Proposition \ref{universal3} that every pair $(E, c_0)$ satisfies that if $K \in \mathcal{K}(E, c_0)$, $T \in \mathcal{L}(E, c_0)$, and $\|T\| < \|T + K\|$, then $T + K$ is norm-attaining. 
However, we do not know whether the following is true or not.

\begin{question}[R. Aron]\label{question:Aron} Suppose that $E$ and $F$ are reflexive Banach spaces with the following property: whenever $T \in \mathcal{L}(E, F)$ and $K \in \mathcal{K}(E, F)$ are such that $\|T\| < \|T + K\|$, we have that $T + K$ attains its norm. Does it then follow that the pair $(E,F)$ has the WMP?
\end{question}
 
\noindent
\textbf{Acknowledgements:} The authors are grateful to Richard Aron, Luis Carlos Garc\'ia Lirola, Jos\'e David Rodr\'iguez Abell\'an, Miguel Mart\'in, and Abraham Rueda Zoca for fruitful conversations on the topic of the paper.

\end{document}